\documentclass[11pt,reqno]{amsart}
\usepackage{a4wide}
\usepackage{hyperref}
\usepackage{color}
\usepackage{cite}
\usepackage{amsfonts,amssymb}
\usepackage{mathrsfs}
\usepackage{graphicx}

\definecolor{darkred}{rgb}{1,0,0} %can change the intensity in [0,1]
\definecolor{green}{rgb}{0,0.6,0}
\definecolor{blue}{rgb}{0,0,1}

\hypersetup{colorlinks,linkcolor=blue,filecolor=green,urlcolor=darkred,
citecolor=green}

\theoremstyle{plain}

\newtheorem{neu}{}[section]
\newtheorem{Cor}[neu]{Corollary}
\newtheorem*{Cor*}{Corollary}
\newtheorem{Thm}[neu]{Theorem}
\newtheorem*{Thm*}{Theorem}
\newtheorem{Prop}[neu]{Proposition}
\newtheorem*{Prop*}{Proposition}
\theoremstyle{definition}
\newtheorem{Lemma}[neu]{Lemma}
\newtheorem*{Rmk*}{Remark}
\newtheorem{Rmk}[neu]{Remark}

\newtheorem*{Ex*}{Example}

\newtheorem*{Qu*}{Question}

\newtheorem{Def}[neu]{Definition}

\theoremstyle{remark}

\theoremstyle{definition}

\newcommand{\x}{\times}

\newcommand{\p}{\partial}
\newcommand{\om}{\omega}

\newcommand{\into}{\hookrightarrow}
\newcommand{\pf}{\longrightarrow}

\newcommand{\N}{{\mathbb{N}}}
\newcommand{\Z}{{\mathbb{Z}}}
\newcommand{\R}{{\mathbb{R}}}
\newcommand{\C}{{\mathbb{C}}}
\newcommand{\Q}{{\mathbb{Q}}}

\renewcommand{\H}{\mathrm{H}}

\newcommand{\lcm}{\mathrm{lcm}}

\newcommand{\Id}{\mathrm{Id}}

\newcommand{\cl}{\mathrm{cl}}

\newcommand{\CF}{\mathrm{CF}}

\newcommand{\RFH}{\mathrm{RFH}}

\newcommand{\Sp}{\mathrm{Sp}}

\newcommand{\Fix}{\mathrm{Fix}\,}

\newcommand{\Crit}{{\rm Crit}}

\newcommand{\sign}{{\rm sign\,}}

\renewcommand{\SS}{\mathcal{S}}
\renewcommand{\AA}{\mathcal{A}}

\newcommand{\KK}{\mathcal{K}}

\newcommand{\EE}{\mathcal{E}}
\newcommand{\MM}{\mathcal{M}}
\newcommand{\LL}{\mathcal{L}}
\newcommand{\RR}{\mathcal{R}}

\newcommand{\RRR}{\mathscr{R}}

\newcommand{\beq}{\begin{equation}}
\newcommand{\beqn}{\begin{equation}\nonumber}
\newcommand{\eeq}{\end{equation}}

\newcommand{\bea}{\begin{equation}\begin{aligned}}
\newcommand{\bean}{\begin{equation}\begin{aligned}\nonumber}
\newcommand{\eea}{\end{aligned}\end{equation}}

\numberwithin{equation}{section}
\numberwithin{figure}{section}
\begin{document}
%%%%%%%%%%%%%%%%%%%%%%%%%%%%%%%%%%%%%%%%%%%%
\title[Symmetric periodic orbits in the restricted three-body problem]{Some remarks on symmetric periodic orbits in the restricted three-body problem}
\author{Jungsoo Kang}
\address{Department of Mathematics\\
     Seoul National University, Seoul, Korea}
\address{Mathematisches Institut,
Westf\"alische Wilhelms-Universit\"at M\"unster, M\"unster, Germany}
\email{jungsoo.kang@me.com}
\begin{abstract}
 The planar circular restricted three-body problem (PCRTBP) is symmetric with respect to the line of masses and there is a corresponding anti-symplectic involution on the cotangent bundle of the 2-sphere in the regularized PCRTBP. Recently it turned out that each bounded component of an energy hypersurface with low energy for the regularized PCRTBP is fiberwise starshaped. This enables us to define a Lagrangian Rabinowitz Floer homology which is related to periodic orbits symmetric for the anti-symplectic involution in the regularized PCRTBP and hence to symmetric periodic orbits in the unregularized problem. We compute this homology and discuss the properties of the symmetric periodic
orbits.
\end{abstract}

\maketitle
\setcounter{tocdepth}{1}

\section{Introduction}
The project to apply holomorphic curve techniques to the planar circular restricted three-body problem (PCRTBP) just began, see \cite{AFvKP12,AFFHvK11,AFFvK12,CFvK11}. In particular, Albers-Frauenfelder-van Koert-Paternain \cite{AFvKP12} proved that each bounded component of the regularized energy hypersurface is  a fiberwise starshaped hypersurface in $T^*S^2$ for energy less than the first critical value. As they mentioned this opens up the possibility of applying holomorphic curve techniques. In this paper we compute a related Lagrangian Rabinowitz Floer homology and using this computation we make various observations about symmetric periodic orbits which we will introduce below. We review the regularized PCRTBP in Section 2 and background material on Lagrangian Rabinowitz Floer homology is contained in Section 3.  \\[-1.5ex]

  We refer to two massive primaries as the earth and the moon and to the other body with negligible mass as the satellite. The configuration space is $\R^2\setminus\{q^E,q^M\}$ and the phase space is given by  $T^*(\R^2\setminus\{q^E,q^M\})=(\R^2\setminus\{q^E,q^M\})\x\R^2$. Here $q^E$ and $q^M$ are taken to lie on the real axis and represent the positions of the earth and the moon respectively. The Hamiltonian for the planar circular restricted three body problem (PCRTBP) is given by 
\beq\label{eq:Hamiltonian for the PCRTBP}
H(q,p)=\frac{1}{2}|p|^2-\frac{\mu}{|q-q^M|}-\frac{1-\mu}{|q-q^E|}+q_2p_1-q_1p_2
\eeq
where $\mu\in(0,1)$ is the normalized mass of the moon. The energy hypersurface $H^{-1}(c)$ with energy $c\in\R$ below the first critical value $H(L_1)$ is composed of three connected components. Following \cite{AFvKP12}, we denote by $\Sigma_c^E$ resp. $\Sigma_c^M$ the bounded component close to the earth resp. to the moon. Since these components are noncompact due to collisions, we compactify each of them into $\overline\Sigma_c^E$ and $\overline\Sigma_c^M$ by Moser regularization. Since the discussions in this paper go through for both $\overline\Sigma_c^E$ and $\overline\Sigma_c^M$ with $c<H(L_1)$, we call them $\Sigma$ for convenience. The regularized phase space is the cotangent bundle of $S^2$ and $\Sigma$ is diffeomorphic to the unit cotangent bundle of $S^2$. We denote the Hamiltonian function corresponding to $H$ via Moser regularization by $Q\in C^\infty(T^*S^2)$. More details can be found in Section 2. \\[-1.5ex]

An interesting feature of the PCRTBP is that there is an involution such that the problem is symmetric with respect to this involution. More precisely, there exists an anti-symplectic involution $\RR$ on $T^*\R^2$ given by 
$$
\RR(q_1,q_2,p_1,p_2)=(q_1,-q_2,-p_1,p_2)
$$
such that $H\circ\RR=H$. Thus the Hamiltonian vector field of $H$ is invariant under $\RR$. Through Moser regularization, $\RR$ induces the anti-symplectic involution 
$$
\RRR=I\circ T^*\rho:T^*S^2\pf T^*S^2
$$ 
where $I:T^*S^2\to T^*S^2$ is given by $I(\xi,\eta)=(\xi,-\eta)$ and $\rho$ is the reflection on $S^2$ about a great circle. The fixed locus of $\RRR$ is the conormal bundle of the great circle. In the present paper we are concerned with a periodic orbit of prescribed energy which is carried into itself by $\RRR$, i.e. $(x,2T)$ satisfying
\beq\label{eq:symmetric periodic Reeb orbit}
x:\R/2T\Z\to\Sigma,\quad\dot x=X_Q(x),\quad x(T+t)=\RRR x(T-t)
\eeq
which we refer to a {\em symmetric periodic orbit}. Here $X_Q$ is the Hamiltonian vector field of $Q$. As mentioned above, $\Sigma$ is shown to be a fiberwise starshaped hypersurface (tight $\R P^3$) in $T^*S^2$ and thus $\Fix\RRR\cap\Sigma$ is diffeomorphic to the disjoint union of two circles $L_+$ and $L_-$ (legendrian knots). We note that every symmetric periodic orbit $(x,2T)$ intersects with $L_+\cup L_-$ exactly twice at time $0$ and $T$ (after an appropriate time shift). It is an interesting question whether a symmetric periodic orbit intersects with both circles or only one of them. 

\begin{Def}
A symmetric periodic orbit on $\Sigma$ is called of {\em type I} if it intersects with both $L_+$ and $L_-$. Otherwise, we call it of {\em type II}.
\end{Def}

For explicit computations, let us consider the case  $\Sigma=\overline\Sigma^M_c$ with $c<H(L_1)$. We embed the cotangent bundle of $S^2$ in $\R^6$ as below.
$$
T^*S^2=\{(\xi,\eta)\in\R^6\,|\,\xi\in S^2,\,\xi\cdot\eta=0\}.
$$
Then the inverse process of Moser regularization gives a correspondence
\bean
\mathcal{M}:T^*S^2&\pf T^*\R^2,\\ (\xi,\eta)&\longmapsto \Big(\eta_1(1-\xi_0)+\xi_1\eta_0+q_1^M,\eta_2(1-\xi_0)+\xi_2\eta_0,\frac{-\xi_1}{1-\xi_0},\frac{-\xi_2}{1-\xi_0}\Big)
\eea
where the moon is located at $q^M=(q_1^M,q_2^M)=(-(-1-\mu),0)$ for $\mu\in(0,1)$.
Via this map, the anti-symplectic involution $\RR$ on $T^*\R^2$ corresponds to the anti-symplectic involution $\RRR$ on $T^*S^2$ defined by 
$$
\RRR(\xi_0,\xi_1,\xi_2,\eta_0,\eta_1,\eta_2)=(\xi_0,-\xi_1,\xi_2,-\eta_0,\eta_1,-\eta_2),\quad (\xi,\eta)\in T^*S^2.
$$
As mentioned, $\RRR$ can be regarded as the composition of two involutions 
$$
I:T^*S^2\to T^*S^2,\quad I(\xi,\eta)=(\xi,-\eta).
$$
and 
$$
T^*\rho:T^*S^2\to T^*S^2,\quad T^*\rho(\xi_0,\xi_1,\xi_2,\eta_0,\eta_1,\eta_2)=(\xi_0,-\xi_1,\xi_2,\eta_0,-\eta_1,\eta_2)
$$
where $\rho$ is the reflection on $S^2$ about the great circle $L=\{\xi\in S^2\subset\R^3\,|\, \xi_1=0\}$. Then the fixed locus of $\RRR$ is the conormal bundle of $L$.
$$
\Fix\RRR=\{(\xi,\eta)\in T^*S^2\,|\,\xi_1=0,\,\eta_0=\eta_2=0\}\cong N^*L.
$$
Since $\Sigma$ is a fiberwise starshaped hypersurface,  $\Fix\RRR\cap\Sigma$ is composed of two circles 
$$
L^M_+:=\{(\xi,\eta)\in\Fix\RRR\,|\,\eta_1=f_+(\xi_0,\xi_2)\},\quad L^M_-:=\{(\xi,\eta)\in\Fix\RRR\,|\,\eta_1=f_-(\xi_0,\xi_2)\}
$$
where $f_\pm:\{(\xi_0,\xi_2)\,|\,\xi_0^2+\xi_2^2=1\}\to\R_\pm$ is a positive/negative function. Let  $\pi:T^*\R^2\to\R^2$ be the footpoint  projection map. Then the regions of $L^M_+$ and $L^M_-$ in the configuration space of the PCRTBP are as below.
$$
\LL_-^M:=\pi\circ\mathcal{M}(L^M_-)=\{(q_1,0)\,|\,q_1\in[\min_{(\xi,\eta)\in\Fix\RRR}\{ f_-(1-\xi_0)\}+q_1^M,q_1^M]\}
$$
and
$$
\LL_+^M:=\pi\circ\mathcal{M}(L^M_+)=\{(q_1,0)\,|\,q_1\in[q_1^M,\max_{(\xi,\eta)\in\Fix\RRR}\{f_+(1-\xi_0)\}+q_1^M]\}. 
$$
The region $\KK_c^M=\pi(\overline{\Sigma}_c^M)$ is called the {\em Hill's region} (around the moon) where the satellite with energy $c$ can move. In Figure \ref{fig:Hill's region1} below we depict $\KK_c^M$ (the region encircled by the dotted curve) and $\LL^M_\pm$.
\begin{figure}[htb]
\includegraphics[width=.65\textwidth,clip]{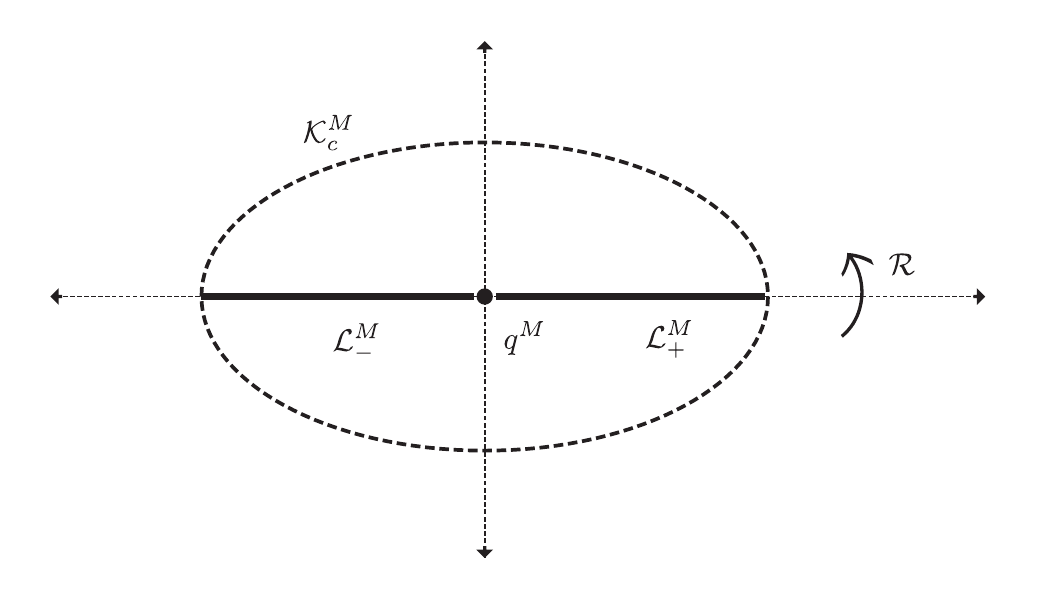}\caption{Hill's region near the moon}\label{fig:Hill's region1}
\end{figure}

Suppose that a symmetric periodic orbit $(x,2T)$ does not pass through the north pole of $S^2$ (i.e. does not collide with the moon). Then there is a periodic solution $((q^x(t),p^x(t)),2T)$ of the Hamiltonian system of  \eqref{eq:Hamiltonian for the PCRTBP} corresponding to $(x,2T)$. Since $(q^x(t),p^x(t))$ passes through $\Fix\RR$ at time $0$ and $T$ and 
$$
\dot q^x_1=\frac{\p H}{\p p_1}=p^x_1+q^x_2=0 \quad\textrm {on}\quad \Fix\RR=\{(q_1,0,0,p_2)\},
$$
$q^x(t)$ cuts the $q_1$-axis at a right-angle at time $0$ and $T$. The figures \ref{fig:Type I} and \ref{fig:Type II} describe the geometric motions of symmetric periodic orbits in the configuration space $\R^2$. If $(x,2T)$ is of type I, $(q_1^x(0)-q_1^M)(q_1^x(T)-q_1^M)<0$, see Figure \ref{fig:Type I}. We note that the Birkhoff retrograde orbit \cite{Bir15} which looks like $\mathcal{X}_1$ is of type I. On the other hand, if $(x,2T)$ is of type II, $(q_1^x(0)-q_1^M)(q_1^x(T)-q_1^M)>0$, see Figure \ref{fig:Type II}. We doubt whether there is a symmetric periodic orbit which does not surround the primary like $\mathcal{X}_3$. But we expect Type II symmetric periodic orbits like $\mathcal{X}_4$ mostly exist in the PCRTBP for arbitrary $c<H(L_1)$ and $\mu\in(0,1)$. Indeed, when $\mu=0$ (the rotating Kepler problem), there always exist such Type II symmetric periodic orbits for every energy below the first critical value: A symmetric periodic orbit which is a $k$-fold covered ellipse in an $l$-fold covered coordinate system (defined in \cite{AFFvK12}) is of type II whenever $k+l$ is odd. Then the perturbation method based on the implicit function theorem, for instance \cite{Are63,Bar65}, ensures survival of them at least for small $\mu\approx0$.

\begin{figure}[htb]
\includegraphics[width=.6\textwidth,clip]{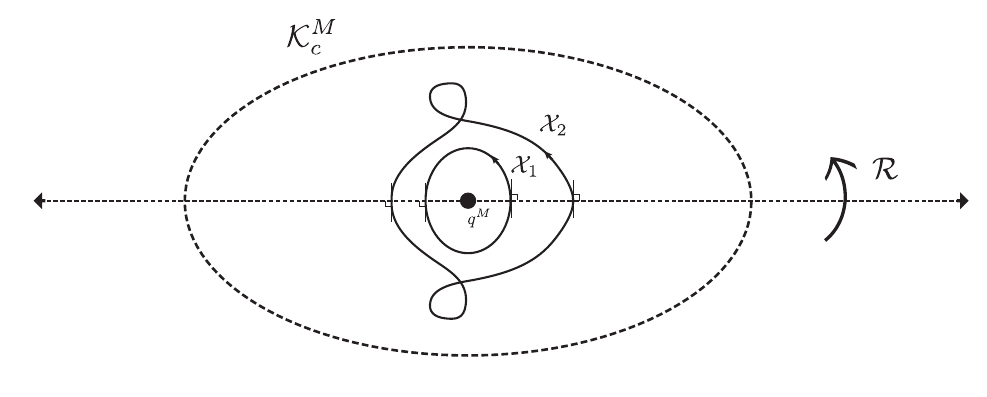}\caption{Type I symmetric periodic orbits}\label{fig:Type I}
\end{figure}

\begin{figure}[htb]
\includegraphics[width=.6\textwidth,clip]{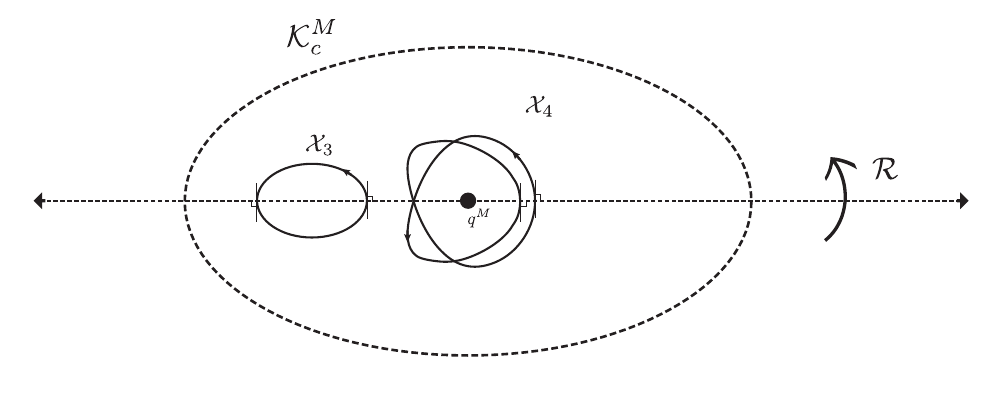}\caption{Type II symmetric periodic orbits}\label{fig:Type II}
\end{figure}
 
The Hamiltonian vector field $X_Q$ on $\Sigma$ can be lifted to the Reeb vector field on a starshaped hypersurface in $\R^4$ with respect to the contact form $\alpha:=\frac{1}{2}(x_1dy_1-y_1dx_1+x_2dy_2-y_2dx_2)$, see e.g. \cite{HP08}. We denote such a double cover of $\Sigma$ by $S\subset\R^4$ and the covering map by $\Pi:S\to\Sigma$. A hypersurface $S\in\R^4$ is called {\em strictly convex} if it represented as $S=H^{-1}(0)$ for $H\in C^\infty(\R^4)$ satisfying $H''(x)\geq c\Id_{\R^4}$, $x\in S$, for some $c>0$.\\[-1.5ex]

\noindent\textbf{Observation A.} {\em Suppose that $S$ is strictly convex.}
\begin{itemize}
\item[(A1)] {\em There exist at least two symmetric periodic orbits on $\Sigma$.}
\item[(A2)] {\em If there are precisely two periodic orbits on $\Sigma$, both are symmetric periodic orbits of type I.}
\item[(A3)] {\em There exist infinitely many periodic orbits on $\Sigma$ if a type II symmetric periodic exists.}
\end{itemize}

\begin{proof}
We note that $S$ is a centrally symmetric hypersurface, i.e. $S=-S$ in $\R^4$ and that there is an anti-symplectic involution $\widetilde\RRR$ on $\R^4$ given by 
$$
\widetilde\RRR(x_1,x_2,y_1,y_2):=(-x_1,x_2,y_1,-y_2)
$$
such that $\Pi\circ\widetilde\RRR|_{S}=\RRR|_{\Sigma}$. Since we have a map $\Psi(x_1,x_2,y_1,y_2):=(x_1,-y_2,y_1,x_2)$ on $\R^4$ 
such that $\Psi^*\alpha=\alpha$, the result on $\RRR$-symmetric periodic orbits on $S$ can be inferred from the result on periodic orbits symmetric with respect to 
$$
N:=\Psi\circ\widetilde\RRR\circ\Psi^{-1},\quad (x_1,x_2,x_3,x_4)\mapsto(-x_1,-x_2,y_1,y_2)
$$
on $\Psi(S)$ which is also centrally symmetric due to $\Psi\circ-\Id_{\R^4}=-\Id_{\R^4}\circ\Psi$. It is worth remarking that $N$-symmetric periodic orbits are well known as  {\em brake orbits} in classical mechanics which have a rich history. In particular, we can employ a theorem of \cite{LZZ06} to prove the assertion (A1) when $S$ is strictly convex. 

The assertions (A2) and (A3) are  immediate consequences of (A1) and a theorem in \cite{HWZ98}, see Remark \ref{rmk:ellipsoid}. Indeed if there is a type II symmetric periodic orbit $(x,2T)$, there are two distinct periodic Reeb orbits $(\tilde x_1,2T)$ and $(\tilde x_2,2T)$ on $S$ such that $\pi(\tilde x_1)=\pi(\tilde x_2)=x$. But there is another periodic Reeb orbit on $S$ due to (A1) and thus the theorem of \cite{HWZ98} guarantees the existence of infinitely many periodic orbits on $S$ and hence on $\Sigma$ as well.
\end{proof}

We summarize the following correspondence which was used in the proof of the observation.

\bean
\textrm{The regularized PCRTBP}\quad\;&\longleftrightarrow &&\!\!\!\textrm{Centrally symmetric reversible $S^3$}\\[.5ex]
\{\textrm{Type I symmetric periodic orbit}\}&\longleftrightarrow &&\!\!\!\!\{\textrm{Centrally symmetric brake orbit}\}\\[.5ex]
\{\textrm{Type II symmetric periodic orbit}\}&\longleftrightarrow &&\!\!\!\!\{\textrm{Centrally asymmetric brake orbit}\}\\[.5ex]
\eea

The above observation can be used to prove the existence of infinitely many periodic orbits in the PCRTBP for a given mass ratio and a given energy level: For such a prescribed data, it is not difficult to check if the energy hypersurface is strictly convex. Moreover it is conceivable that in some cases a type II symmetric periodic orbit can be detected by a numerical method. Aforementioned type II symmetric periodic orbits exist indeed for small $\mu\approx0$. A recent paper \cite{AFFHvK11} shows that $S$ is strictly convex for large $\mu\approx1$. Contrary to expectations, $S$ is not strictly convex for $\mu\approx0$,  but dynamically convex (see Theorem B) which is a symplectical generalization of strictly convexity, see \cite{AFFvK12}.
\begin{Rmk}\label{rmk:ellipsoid}
In order to give examples where (A2) holds, we consider the ellipsoid
$$
\EE_{4}(r_1,r_2)=\Bigg\{(z_1,z_2)\in\C^2\;\Big|\;\frac{|z_1|^2}{r_1}+\frac{|z_2|^2}{r_2}=1,\;a_2\geq a_1>0\Bigg\}
$$
which is a typical example of a strictly convex hypersurface in $(\C^2,\frac{i}{2}(dz_1\wedge d\bar z_1+dz_2\wedge d\bar z_2))$. The Reeb flow on the ellipsoid $\EE_{4}(r_1,r_2)$ is given by
$$
z(t)=(z_1,z_2)(t)=(a_1e^{2ti/r_1},a_2e^{2ti/r_2}),\quad \frac{a_1^2}{r_1}+\frac{a_2^2}{r_2}=1,\;a_1,\,a_2>0.
$$
The minimal periods of $z_1(t)$ and $z_2(t)$ are $T_1=\pi r_1$ and $T_2=\pi r_2$ respectively. Thus if $r_1/r_2\notin\Q$, there are precisely two centrally symmetric periodic orbits $(\sqrt r_1e^{2ti/r_1},0)$ and $(0,\sqrt r_2e^{2ti/r_2})$. In contrast, if $r_1/r_2\in\Q$, all orbits are periodic with the minimal period $T=\lcm(p,q)T_1/p=\lcm(p,q)T_2/q$ where $p,\,q\in\N$ satisfy $p/q=r_1/r_2$. 

This shows that the ellipsoid possesses either two or infinitely many periodic orbits. In fact this dichotomy remains true for a wider class of 3-dimensional starshaped hypersurfaces: dynamically convex starshaped hypersurfaces \cite{HWZ98}; see also \cite{HWZ03}.
\end{Rmk}

 The next result of the present note is 1) to compute the Lagrangian Rabinowitz Floer homology which is relevant to symmetric periodic orbits in the PCRTBP (more generally, tight $\R P^3$s) and 2) to reprove Observation A using this computation in the nondegenerate and dynamically convex case. 
 
 We refer to $S$ {\em dynamically convex} if the Conley-Zehnder index (defined in Section 3) of every periodic Reeb orbit is greater than or equal to 3. It was proved that a strictly convex hypersurface is $\R^4$ is dynamically convex, see \cite[Theorem 3.4]{HWZ98} or \cite[Chapter 15]{Lon02}. Therefore $S$ is dynamically convex when $\mu\approx 0$ or $\mu\approx1$ due to \cite{AFFHvK11,AFFvK12}.\\[-1.5ex]

\noindent\textbf{Theorem B.} {\em Suppose that $S$ is nondegenerate and dynamically convex. Then,
\beq\label{eq:computation of LagRFH}
\RFH_*(\Sigma,\Fix\RRR,T^*S^2)=\Z_2\oplus\Z_2\oplus\Z_2\oplus\Z_2,\quad *\in\Z\setminus\{0,1\}.
\eeq
Moreover from this computation we can derive Observation A.}\\[-1.5ex]

The homology computation \eqref{eq:computation of LagRFH} will be carried out in Theorem \ref{thm:computation of LagRFH2} and the remaining assertion will be proved in subsection 3.4. We close this introductory section with possible applications of Theorem B. First, making use of an idea behind of the theorem, it is possible to find two brake orbits on dynamically convex hypersurfaces in $\R^{2n}$. Another possible application\footnote{To be honest, for this application we need to use Wrapped Floer homology which can be thought as half of Lagrangian Rabinowitz Floer homology.} is as follows. The homology computation \eqref{eq:computation of LagRFH} can be used to guarantee the existence of a gradient flow line which is a solution of a perturbed holomorphic curve equation. Then by pushing out the perturbed term to infinity, we obtain a symmetric punctured holomorphic sphere in $T^*S^2$ which has a finite energy in the sense of Hofer \cite{Hof93}.  Unlike gradient flow lines, holomorphic curves always have positive intersection numbers which are related to the positivity of linking numbers of asymptotic orbits. Moreover we expect that this holomorphic curve gives rise to an annulus type global surface of section in some nice situations. This kind of argument will be treated in the forthcoming paper \cite{FK14}.

%%%%%%%%%%%%%%%%%%%%%%%%%%%%%%%%%%%%%%%%%%%%%%%%
\section{The regularized restricted three body problem}
Though the content of this section can be found in \cite{AFvKP12}, we briefly review the regularized PCRTBP to make this paper self-contained. As the name of the PCRTBP indicates, we assume that the moon and the earth rotate in a circular trajectory with center at the center of masses and that the satellite is massless and moves on the plane where the moon and the earth rotate. Let $m_E$ be the mass of the earth and $m_M$ be the mass of the moon. We denote the normalized mass of $m_M$ by $\mu$, i.e. 
$$
\mu=\frac{m_M}{m_E+m_M}\in[0,1].
$$
In the rotating coordinate system, the earth and the moon are located at
$$
q^E=(\mu,0),\quad q^M=(-(1-\mu),0)
$$
respectively and the phase space is given by 
$$
T^*(\R^2\setminus\{q^E,q^M\})=(\R^2\setminus\{(q^E,q^M)\})\x\R^2
$$ 
The positions of the earth and the moon are removed to avoid collisions. The Hamiltonian for the satellite $H:(\R^2\setminus\{(q^E,q^M)\})\x\R^2\to\R$ is given by
\beqn
H(q,p)=\frac{1}{2}|p|^2-\frac{1-\mu}{|q-q^E|}-\frac{\mu}{|q-q^M|}+q_1p_2-q_2p_1.
\eeq
The Hamiltonian $H$ carries exactly five critical points $(L_1,L_2,L_3,L_4,L_5)$ called Lagrange points. We may assume that 
$$
H(L_1)<H(L_2)\leq H(L_3)<H(L_4)=H(L_5).
$$ 
The Hamiltonian $H$ is invariant under the anti-symplectic involution 
$$
\RR:T^*\R^2\to T^*\R^2,\quad (q_1,q_1,p_1,p_2)\mapsto(q_1,-q_2,-p_1,p_2),
$$
which preserves the three (colinear) Lagrange points and interchanges the two (equilateral) Lagrange points. We denote by $\pi:T^*\R^2\to\R^2$ the footpoint projection map. Then the Hill's region where the satellite with energy $c$ moves is $\pi(H^{-1}(c))$. The region $\pi(H^{-1}(c))$ for $c<H(L_1)$ is composed of two bounded regions and one unbounded region. We abbreviate the bounded regions by $\KK_c^E$ and $\KK_c^M$ so that $q^E\in\cl(\KK_c^E)$ and $q^M\in\cl(\KK_c^M)$. Likewise $H^{-1}(c)$ consists of two bounded components and one unbounded component. We denote by $\Sigma_c^E$ resp. $\Sigma_c^M$ the bounded component corresponding to $\KK_c^E$ resp. $\KK_c^M$.\\[-1.5ex]

In \cite{Mos70}, Moser regularized an energy hypersurfaces of the Kepler problem with negative energy into the unit tangent bundle of $S^2$. The PCRTBP (in the rotating coordinate system) can also be regularized in a similar way, see \cite[Section 6]{AFvKP12}. In what follows we briefly outline the regularization process for $\Sigma^M_c$ which is the bounded component close to the moon. We first introduce an independent variable 
$$
s=\int\frac{dt}{|q-q^M|}
$$
and define the Hamiltonian $K(q,p)$ by
$$
H(q,p)=\frac{K(q,p)}{|q-q^M|}+c.
$$
Here $H^{-1}(c)$ is the energy hypersurface to be regularized. One can easily check that the Hamiltonian flow of $K$ at energy level $0$ with time parameter $s$ corresponds to the Hamiltonian flow of $H$ at energy level $c\in\R$ with time parameter $t$. We set $p=-x$, $q-q^M=y$ and perform the inverse of the stereographic projection. Here the stereographic projection 
\bean
\SS:T^*S^2=\{(\xi,\eta)\}&\pf T^*\R^2=\{(x,y)\}\\ 
\eea
is given by
$$
\SS(\xi,\eta)=\Big(\frac{\xi_1}{1-\xi_0},\frac{\xi_2}{1-\xi_0},\eta_1(1-\xi_0)+\xi_1\eta_0,\eta_2(1-\xi_0)+\xi_2\eta_0\Big)
$$
where $\xi=(\xi_0,\xi_1,\xi_2)\in S^2\subset\R^3$ and $\eta=(\eta_0,\eta_1,\eta_2)\in T^*_\xi S^2$, i.e. $\xi_0\eta_0+\xi_1\eta_1+\xi_2\eta_2=0$. Then we obtain the Hamiltonian function $K\circ\SS$ on $T^*S^2$. Since $K\circ\SS$ is not smooth at the zero section, we consider instead 
$$
Q:T^*S^2\to\R,\quad Q:=\frac{1}{2}|\eta|^2(K\circ\SS+\mu)^2.
$$
Then one can readily check that Hamiltonian vector fields on $\Sigma_c^M$ are lifted to those of $Q^{-1}(\mu^2/2)=:\overline{\Sigma}_c^M$. In a similar vein we can compactify $\Sigma_c^E$ into $\overline{\Sigma}_c^E$. A remarkable theorem of Albers-Frauenfelder-van Koert-Paternain \cite{AFvKP12} asserts that $\overline{\Sigma}^E_c$ and $\overline{\Sigma}_c^M$ are fiberwise starshaped hypersurfaces in $T^*S^2$.\\[-1.5ex]

\section{Rabinowitz Floer homology and Proof of Theorem A}

A simple observation shows that the problem \eqref{eq:symmetric periodic Reeb orbit} can be interpreted as the boundary value problem for the Lagrangian submanifold $\Fix\RRR$. Indeed if $(x,T)$ solves
\beq\label{eq:characteristic chord eq with conormal boundary condition}
x:[0,T]\to\Sigma,\quad\dot x=X_Q(x),\quad (x(0),x(T))\in \Fix\RRR\x\Fix\RRR,
\eeq
then so does $x_{\RRR}(t):=\RRR x(T-t)$ and we obtain a symmetric periodic orbit
$$
x_{\RRR}\# x:\R/2T\Z\to\Sigma,\quad x_{\RRR}\# x(t):=\left\{\begin{aligned} x(t)\qquad &t\in[0,T],\\[.5ex]
x_{\RRR}(t)\;\quad &  t\in [T,2T].
\end{aligned}\right.
$$
As mentioned in the introduction, $\Sigma$ can be both $\overline{\Sigma}_c^E$ and $\overline{\Sigma}_c^M$ for $c<H(L_1)$. 
\subsection{Construction of Lagrangian Rabinowitz Floer homology}\quad\\[-1.5ex]

We briefly introduce Lagrangian Rabinowitz Floer homology and refer to \cite{Mer10,Mer11} for further details. We also refer the reader to \cite[Appendix A]{Fra04} for Morse-Bott homology and Floer's celebrated papers \cite{Flo88a,Flo88b}. Let $M$ be a closed $n$-dimensional manifold and $Q$ be a closed $d$-dimensional submanifold in $M$. We denote by $T^*M$ the cotangent bundle of $M$ and by $N^*Q$ the conormal bundle of $Q$. We note that $N^*Q$ is an exact Lagrangian submanifold in $(T^*N,d\lambda)$ where $\lambda$ is the Liouville 1-form. We denote by
$$
P_{N^*Q}T^*M:=\{x\in C^\infty([0,1],T^*M)\,|\,x(0),\,x(1)\in N^*Q\}.
$$
Let $H\in C^\infty(T^*M)$ be such that $H^{-1}(0)$ is a smooth fiberwise starshaped hypersurface. Then since $H^{-1}(0)$ splits $T^*M$ into one bounded component and one unbounded component, we can modify $H$ to be constant near infinity. For notational convenience we write again $H$ for the modified Hamiltonian function. Then the Rabinowitz action functional $\AA^H:P_{N^*Q}T^*M\x\R\to\R$ is defined by
$$
\AA_H(x,\eta)=-\int_0^1x^*\lambda-\eta\int_0^1H(x)dt.
$$
A critical point $(x,\eta)\in\Crit\AA^H$ satisfies
$$
\dot x=\eta X_H(x(t)),\quad x(t)\in H^{-1}(0),\quad t\in[0,1]
$$
where $X_H$ is the {\em Hamiltonian vector field} associated to $H$ defined implicitly by $i_{X_H}d\lambda=dH$. Thus if $(x,\eta)$ is a nontrivial critical point of $\AA^H$, i.e. $\eta\neq0$, $(x_\eta,\eta)$ where $x_\eta(t):=x(t/\eta)$ solves 
\beq\label{eq:characteristic chord eq}
x_\eta:[0,\eta]\to H^{-1}(0),\quad\dot x_\eta=X_H(x_\eta),\quad (x_\eta(0),x_\eta(\eta))\in N^*Q\x N^*Q.
\eeq
We choose an $d\lambda$-compatible almost complex structure $J$ and define a  metric $m_J$ on $P_{N^*Q}T^*M\x\R$ by
$$
m_J{(x,\eta)}\big[(\hat x_1,\hat\eta_1),(\hat x_2,\hat\eta_2)\big]:=\int_0^1d\lambda(\hat x_1, J(x)\hat x_2)dt+\hat\eta_1\hat\eta_2.
$$
for $(x,\eta)\in P_{N^*Q}T^*M\x\R$ and for $(\hat x_1,\hat\eta_1),\,(\hat x_2,\hat\eta_2)\in T_{(x,\eta)}P_{N^*Q}T^*M\x\R$. Then a map $w\in C^\infty(\R,P_{N^*Q}T^*M\x\R)$ which solves 
\beqn%\label{eq:gradient flow eq}
\p_s w+\nabla_{m_J}\AA^H(w(s))=0
\eeq
is called a {\em gradient flow line} of $\AA^H$ with respect to the metric $m_J$.

A solution $(x,T)$ for $T\neq0$ of \eqref{eq:characteristic chord eq} is called {\em nondegenerate} if 
$$
\dim \big(T\phi_H^T[T_{x(0)}N^*Q]\cap T_{x(T)}N^*Q\big)=0,
$$
or equivalently if $\phi^T_H(N^*Q)$ transversally intersects $N^*Q$ at $x(T)$ where $\phi_H^t$ is the flow of $X_H$. %So we call $(x,\eta)\in\Crit\AA^H$ is nondegenerate if $\Sigma\pitchfork N^*Q$ and $(x_\eta,\eta)$ for $\eta\neq0$ is nondegenerate in the above sense. 
A pair $(H^{-1}(0),N^*Q)$ is also called {\em nondegenerate} if $H^{-1}(0)\pitchfork N^*Q$ and every nontrivial solution of \eqref{eq:characteristic chord eq} is nondegenerate. From now on we assume that $(H^{-1}(0),N^*Q)$ is nondegenerate. Moreover a (trivial) critical manifold $(H^{-1}(0)\cap N^*Q,0)$ of $\AA^H$ is Morse-Bott,  see \cite[Lemma 2.5]{Mer10}. In order to define the Morse-Bott homology of $\AA^H$ we pick an auxiliary Morse-Smale pair $(f,g)$ where $f\in C^\infty(\Crit\AA^H)$ and $g$ is a Riemannian metric on $\Crit\AA^H$. The index for critical points of $f$ is defined by 
$$
\mu_{RFH}:\Crit f\to\Z,\quad \mu_{RFH}(x,\eta):=\left\{\begin{aligned}&\sign(\eta)\mu_{RS}(x_\eta,\eta)+d-\frac{n-1}{2},\quad &&\eta\neq 0,\\
&d-n+1+i_{f}(x,0),&&\eta=0.\end{aligned}\right.
$$
where $\mu_{RS}$ is the transverse Robbin-Salamon index defined in \eqref{eq:Robbin-Salamon index} and $i_f$ is the Morse index for $f$, i.e. the number of negative eigenvalues of the Hessian of $f$. %\footnote{ Here, we assume for a technical reason that $c_1|_{\pi_2(T^*M,N^*Q)}$ vanishes. The degree shifting $d-\frac{n-1}{2}$ is due to Theorem \ref{thm:computation of LagRFH1}.} 
We denote by $\Crit_q f$ the set of critical points of $f$ with RFH-index $q\in\Z$. We define a $\Z/2$-vector space
$$
\CF_q(\AA^H,f):=\Big\{\sum_{c\in\Crit_qf}\xi_cc\;\Big|\;\xi_c\in\Z/2\,\Big\}
$$
with the finiteness condition
$$
\#\{c\in\Crit_qf\,|\,\xi_c\neq0,\;\AA^H(c)\geq\kappa\}<\infty.
$$
Next we recall the Frauenfelder's Morse-Bott boundary operator, namely counting gradient flow lines with cascades. For $c_-,c_+\in\Crit f$ and $m\in\N$, a {\em flow line from $c_-$ to $c_+$ with $m$ cascades} 
$$
(\textbf{w},\textbf{t})=\big((w_i)_{1\leq i\leq m},(t_i)_{1\leq i\leq m-1}\big)
$$
consists of gradient flow lines $w_i\in C^\infty(\R,P_{N^*Q}T^*M\x\R)$ of $\AA^H$ and positive real numbers $t_i\in\R_+$ such that 
$$
\lim_{s\to\infty} (w_1(-s),w_m(s))\in W^u(c_-;f)\x W^s(c_+;f),\quad
\lim_{s\to-\infty} w_{i+1}(s)=\phi_f^{t_i}(\lim_{s\to\infty}w_i(s))
$$
for $i=1,\dots,m-1$. Here $W^u(c_-;f)$ resp. $W^s(c_+;f)$ is the unstable manifold resp. the stable manifold and $\phi_f^t$ is the flow of $-\nabla_g f$. It is noteworthy that a flow line with no cascades is nothing but an ordinary negative gradient flow line of $f$. We denote by $\widehat\MM_m(c_-,c_+)$ the space of flow lines with $m$ cascades from $c_-$ to $c_+$. We divide out the $\R^m$-action on $\widehat\MM_m(c_-,c_+)$ defined by shifting the $m$ cascades in the $s$-variable. Then we obtain gradient flow lines with unparametrized cascades and abbreviate  $\MM_m(c_-,c_+):=\widehat\MM_m(c_-,c_+)/\R^m$. We define {\em the set of flow lines with cascades from $c_-$ to $c_+$} by
$$
\MM(c_-,c_+):=\bigcup_{m\in\N\cup\{0\}}\MM_m(c_-,c_+).
$$
The standard arguments in Floer theory proves the following nontrivial facts. For a generic almost complex structure $J$ and a generic Riemannian metric $g$,
\begin{enumerate}
\item[(F1)] $\MM(c_-,c_+)$ is a smooth manifold of finite dimension $\mu_{RFH}(c_-)-\mu_{RFH}(c_+)-1$. Moreover, if $\mu_{RFH}(c_-)-\mu_{RFH}(c_+)=1$, $\MM(c_-,c_+)$ is a finite set.
\item[(F2)] Let $\MM_c(c_-,c_+)$ be the compactification of $\MM(c_-,c_+)$ with respect to the topology of Floer-Gromov convergence. If $\mu_{RFH}(c_-)-\mu_{RFH}(c_+)=2$, $\MM_c(c_-,c_+)$ is a compact one-dimensional manifold whose boundary is 
$$
\p\MM_c(c_-,c_+)=\bigcup_{z}\MM(c_-,z)\x\MM(z,c_+)
$$
where the union runs over $z\in\Crit f$ with $\mu_{RFH}(c_-)-1=\mu_{RFH}(z)$.
\end{enumerate}
Due to (F1), we denote by $n(c_-,c_+)$ the parity of the finite set $\MM(c_-,c_+)$ when $\mu_{RFH}(c_-)-\mu_{RFH}(c_+)=1$. Then the boundary operators $\{\p_q\}_{q\in\Z}$ are defined by
\bean
\p_q:\CF_q(\AA^H,f)&\pf\CF_{q-1}(\AA^H,f)\\
c_-\in\Crit_qf&\longmapsto\!\!\!\!\!\sum_{c_+\in\Crit_{q-1}f}\!\!\!\!\!n(c_-,c_+)\cdot c_+.
\eea
(F2) yields that $\p_{q-1}\circ\p_{q}=0$ and $(\CF_*(\AA^H,f),\p_*)$ is indeed a chain complex. Thus we define {\em Lagrangian Rabinowitz Floer homology} by
$$
\RFH_q(H^{-1}(0),N^*Q,T^*M):=\H_q(\CF_*(\AA^H,f),\p_*).
$$
As the above notation indicates, Lagrangian Rabinowitz Floer homology is invariant under the choice of $(H,J,f,g)$ and depends only on $(H^{-1}(0),N^*Q,T^*M)$.

\subsection{Computation of Lagrangian Rabinowitz Floer homology}\quad\\[-1.5ex]

Making use of the Abbondandolo-Schwartz short exact sequence in \cite{AS09}, Merry proved the following theorem in \cite[Theorem B]{Mer10} (see also Remark 7.7 and Remark 12.6 in \cite{Mer11}). We should mention that he proved more general statements.

\begin{Thm}\label{thm:computation of LagRFH1}
Let $Q$ and $M$ be closed manifolds with $d\leq n/2$. Then
$$
\RFH_*(\Sigma,N^*Q,T^*M)\cong \H_*(P_{Q}M;\Z_2)\oplus \H^{-*+2d-n+1}(P_{Q}M;\Z_2),\quad *\in\Z\setminus\{0,1\}
$$
where the path space $P_{Q}M$ is defined below.
\end{Thm}
\begin{Rmk}
Although in proving the above theorem one has to use a Hamiltonian function defining $\Sigma$ which has quadratic growth \cite{AS09,Mer10} or linear growth \cite{CFO10} near infinity, the resulting Floer homology coincides with the Rabinowitz Floer homology defined in the previous subsection, see \cite[Section 3]{AS09} and \cite[Section 4]{CFO10}.
\end{Rmk}

In what follows we compute the singular homology groups in Theorem \ref{thm:computation of LagRFH1} in a special case. Let $Z$ be a closed connected manifold and $Y$ be a connected submanifold. We denote by $\Omega Z$ the based loop space of $Z$. We further abbreviate relative path spaces for $z,\,z'\in Z$,
\bean
P_{z,z'}Z&:=\big\{\gamma\in C^0([0,1],Z)\,|\,(\gamma(0),\gamma(1))=(z,z')\big\},\\
P_{z,Y}Z&:=\big\{\gamma\in C^0([0,1],Z)\,|\,(\gamma(0),\gamma(1))\in \{z\}\x Y\big\},\\
P_{Y}Z&:=\big\{\gamma\in C^0([0,1],Z)\,|\,(\gamma(0),\gamma(1))\in Y\x Y\big\}.
\eea
Here we deal with continuous paths but the homotopy types of the above path spaces do not change  if we consider $W^{1,2}$-, or $C^\infty$-paths instead, see \cite{Pal66,Kli78}. Suppose that $Y$ is contractible to $z\in Z$ in $Z$; that is, there exists a continuous map $F:Y\x I\to Z$ such that 
$$
F(\cdot,0)=\bar z,\quad F(\cdot,1)=i_Y
$$
where $\bar z:Y\to\{z\}$ is a constant map and $i_Y:Y\to Z$ is a canonical inclusion map.

\begin{Prop}\label{prop:homotopy equivalences}
Let $Y$ be contractible to $z\in Z$ in $Z$ as above. Then we have the following homotopy equivalences:
$$
P_YZ\simeq P_{z,Y}Z\x Y\simeq \Omega Z\x Y\x Y.
$$
\end{Prop}
\begin{proof}
We define for each $y\in Y$, $\gamma^y$ a path in $Z$ by
$$
\gamma^y(t):=F(y,t),\quad t\in [0,1].
$$
In particular, $\gamma^y(0)=z$ and $\gamma^y(1)=y$. We set
$$
\bar\gamma^y(t):=\gamma^y(1-t),\quad \gamma_r^y(t):=\gamma^y(rt),\quad r\in[0,1].
$$
We define a map $\Phi$ which will give the desired homotopy equivalence. Here we abbreviate $\#$ for the concatenation operation for paths.
\bean
\Phi:P_{z,Y}Z&\pf \Omega Z\x Y\\
u&\longmapsto \big(\bar\gamma^{u(1)}\#u,u(1)\big)
\eea
The map $\Psi$ below will be a homotopical inverse of $\Phi$.
\bean
\Psi:\Omega Z\x Y&\pf P_{z,Y}Z\\
(w,y)&\longmapsto \gamma^y\#w
\eea
Here we consider $\Omega Z$ as a loop space of $Z$ with the base point $z\in Z$. In order to show that $\Psi\circ\Phi$ is homotopic to the identity, we construct a homotopy 
\bean
G:P_{z,Y}Z\x [0,1]&\pf P_{z,Y}Z\\
(u,r)&\longmapsto \gamma_r^{u(1)}\#\bar\gamma_r^{u(1)}\#u
\eea
such that 
$$
G(u,0)=
\left\{\begin{array}{ll}
u(3t)\quad& 0\leq t\leq1/3\\[.5ex]
u(1)& 1/3\leq t\leq 1
\end{array}\right.
,\qquad G(\cdot,1)=\Psi\circ\Phi.
$$
Performing some reparametrizations on $G$ at time $r=0$, we deduce
$$
\Psi\circ\Phi\simeq\Id_{P_{z,Y}Z}.
$$
In a similar vein, using the homotopy 
\bean
R:\Omega Z\x Y\x [0,1]&\pf \Omega Z\x Y\\
(w,y,r)&\longmapsto (\bar\gamma_r^{y}\#\gamma_r^{y}\#w,y)
\eea
such that 
$$
R(w,y,0)=
\left\{\begin{array}{ll}
\big(w(3t),y\big)\quad& 0\leq t\leq1/3\\[.5ex]
\big(w(1),y\big)& 1/3\leq t\leq 1
\end{array}\right.
,\qquad R(\cdot,\cdot,1)=\Phi\circ\Psi,
$$
we obtain after some reparametrizations as before,
$$
\Phi\circ\Psi\simeq\Id_{\Omega Z\x Y}.
$$
This proves $P_{z,Y}Z\simeq \Omega Z\x Y$ and thus the second equivalence is proved. The first equivalence  $P_YZ\simeq P_{z,Y}Z\x Y$ follows  analogously.
\end{proof}

\begin{Cor}\label{cor:homology of P_{S^1}(S^2)}
Let $S^1$ be an embedded circle in $S^2$. Then we have
$$
P_{S^1}S^2\simeq\Omega S^2\x S^1\x S^1.
$$
In particular, we compute
\beqn
\H_n(P_{S^1}(S^2);\Z_2)=
\left\{\begin{array}{ll}
\Z_2\qquad & n=0,\\[.5ex]
\Z_2\oplus\Z_2\oplus\Z_2 & n=1,\\[.5ex]
\Z_2\oplus\Z_2\oplus\Z_2\oplus\Z_2 & \textrm{otherwise}.
\end{array}\right.
\eeq
\end{Cor}

\begin{Rmk}
In an alternative way, one can directly compute the singular homology of $P_{S^1}S^2$ by means of the {\em Leray-Serre spectral sequence}. We consider the evaluation map $ev_1:P_{z,S^1}S^2\to S^1$ defined by $ev_1(u)=u(1)$. Then we have a fibration 
$$
\Omega S^2\into P_{z,S^1}S^2\stackrel{ev_1}{\pf} S^1.
$$
We note that the spectral sequence for this fibration degenerates at the second page for dimension reasons, i.e. $E^\infty=E^2$. Even though $S^1$ is not simply-connected, the $E^2$-page has a simple formula. Since $S^1$ is contractible in $S^2$, the above fibration has trivial monodromy
$$
\pi_1(S^1)\to\mathrm{Aut}\big(\H_n(\Omega S^2)\big),\quad \ell\mapsto\Id_{\H_n(\Omega S^2)},\quad \forall\ell\in\pi_1(S^1),\;n\in\N\cup\{0\},
$$
and thus 
$$
E^2_{i,j}\cong\H_i\big(S^1;H_j(\Omega S^2;\Z_2)\big)\cong\H_i(S^1;\Z_2)\otimes \H_j(\Omega S^2;\Z_2).
$$
Therefore we have 
$$
\H_n(P_{z,S^1}S^2;\Z_2)\cong\bigoplus_{i+j=n} E_{i,j}^\infty\cong \bigoplus_{i+j=n}\H_i(S^1;\Z_2)\otimes\H_j(\Omega S^2;\Z_2).
$$
Then exactly the same arguments go through for a fibration 
$$
P_{z,S^1}S^2\into P_{S^1}S^2\stackrel{ev_0}{\pf} S^1
$$
where $ev_0$ is the evaluation map at time zero. Therefore we derive 
$$
\H_n(P_{S^1}S^2;\Z_2)\cong\bigoplus_{i+j+k=n} \H_i(S^1;\Z_2)\otimes \H_j(S^1;\Z_2)\otimes \H_k(\Omega S^2;\Z_2).
$$
\end{Rmk}

Therefore Theorem \ref{thm:computation of LagRFH1} and Proposition \ref{prop:homotopy equivalences} result in the following.
\begin{Thm}\label{thm:computation of LagRFH2}
Let $Q$ and $M$ be as above and $Q$ be contractible to a point in $M$.
\bean
\RFH_*(\Sigma,N^*Q,T^*M)&=\bigoplus_{*_1+*_2+*_3=*}\big(\H_{*_1}(\Omega M,\Z_2)\otimes\H_{*_2}(Q;\Z_2)\otimes\H_{*_3}(Q;\Z_2)\big)\\
&\bigoplus_{\star_1+\star_2+\star_3=-*+2d-n+1}\!\!\!\!\!\!\!\!\!\!\!\!\big(\H^{\star_1}(\Omega M;\Z_2)\otimes\H^{\star_2}(Q;\Z_2)\otimes\H^{\star_3}(Q;\Z_2)\big).
\eea
In particular if $S^1$ is an embedded circle in $S^2$,
\beqn
\RFH_*(S^*S^2,N^*S^1,T^*S^2)=\Z_2\oplus\Z_2\oplus\Z_2\oplus\Z_2,\quad *\in\Z\setminus\{0,1\}.
\eeq
\end{Thm}

\subsection{Robbin-Salamon index}\quad\\[-1.5ex]

We denote by $\LL(\R^{2n})$ the Grassmanian manifold of all Lagrangian subspaces in $(\R^{2n},\om_0=dx\wedge dy)$. Let $V\in\LL(\R^{2n})$ and $\Lambda:[0,T]\to\LL(\R^{2n})$. We choose $W\in\LL(\R^{2n})$ a Lagrangian complement of $\Lambda(t)$ for $t\in[0,T]$. For $v\in\Lambda(t)$ and small $\epsilon$ we can find a unique $w(\epsilon)\in W$ such that $v+w(\epsilon)\in\Lambda(t+\epsilon)$. The {\em crossing form} at time $t\in[0,T]$ is defined by
\bean
\Gamma(\Lambda,V,t):\Lambda(t)\cap V&\pf\R\\
v&\longmapsto \frac{d}{d\epsilon}\Big|_{\epsilon=0}\om_0(v,w(t+\epsilon)).
\eea
It is independent of the choice of $W$. A crossing time $t\in[0,T]$, i.e. $\Lambda(t)\cap V\neq\{0\}$, is said to be regular if $\Gamma(\Lambda,V,t)$ is nondegenerate. Since regular crossings are isolated, the number of crossings for a regular path which has only regular crossings is finite.  Thus for a regular path $\Lambda(t)\in\LL(\R^{2n})$ and $V\in\LL(\R^{2n})$, the {\em Robbin-Salamon index} \cite{RS93} can be defined as below.
$$
\mu_{RS}(\Lambda,V):=\frac{1}{2}\sign\Gamma(\Lambda,V,0)+\sum_{0<t<T}\sign\Gamma(\Lambda,V,t)+\frac{1}{2}\sign\Gamma(\Lambda,V,T)
$$
where the sum is taken over all crossings $t\in[0,T]$ and $\sign$ denotes the signature of the crossing form. Since we can always perturb a Lagrangian path to be regular and the Robbin-Salamon index is invariant under homotopies with fixed end points, the Robbin-Salamon index for nonregular paths also can be defined. Let $\Psi:[0,T]\to\Sp(2n)$ be a path of symplectic matrices with $\Psi(0)=\Id_{\R^{2n}}$ and $\det(\Id_{\R^{2n}}-\Psi(T))\neq0$. Then the {\em Conley-Zehnder index} of $\Psi$ is defined by
$$
\mu_{CZ}(\Psi):=\mu_{RS}(\mathrm{graph\,}\Psi,\Delta)
$$
where $\Delta$ is the diagonal of $\R^n\x\R^n$.

Returning to the regularized PCRTBP, let $(x,T)$ be a solution of \eqref{eq:characteristic chord eq with conormal boundary condition}. 
Then $(x_\RRR,T)$ and $(x^m,mT)$ for $m\in\N$ defined by
\beqn
x_\RRR(t):=\RRR x(T-t),\quad x^{2k}:=\underbrace{x_\RRR\#\cdots \#x\#x_\RRR\#x}_{2k},\quad x^{2k+1}:=\underbrace{x\#\cdots\#x\# x_\RRR\#x}_{2k+1}
\eeq
solve \eqref{eq:characteristic chord eq with conormal boundary condition} as well. Now we associate a  Robbin-Salamon index to each solution of \eqref{eq:characteristic chord eq with conormal boundary condition}. We first symplectically trivialize the hyperplane field $\ker\lambda_{\Sigma}$ (contact structure) by a pair of global sections 
$$
X_1=(\frac{\xi\x\eta}{|\xi\x\eta|}-\frac{(\xi\x\eta)\cdot n_\eta}{|\xi\x\eta|\,\eta\cdot n_\eta}\eta)\cdot\p_\eta,\quad X_2=-\frac{(\xi\x\eta)\cdot n_\xi}{|\xi\x\eta|\,\eta\cdot f_\eta}\eta\cdot\p_\eta+\frac{(\xi\x\eta)}{|\xi\x\eta|}\cdot\p_\xi
$$ 
where $n_\Sigma=(n_\xi\p_\xi,n_\eta\p_\eta)$ is the outward pointing normal vector field on $\Sigma$. We denote the induced global symplectic trivialization by
$$
\Phi(\zeta):\ker\lambda_{\zeta}\pf \R^2,\quad \zeta=(\xi,\eta)\in\Sigma.
$$
We note that $\Phi$ is a vertical preserving symplectization trivialization and maps $T\RRR$ to the reflection about the $X_2$-axis, i.e.
$$
\Phi(\zeta)[T_\zeta^vT^*S^2\cap \ker\lambda_\zeta]=(0)\x\R
$$
where $T^vT^*S^2:=\ker T\pi$ be the vertical subbundle of $TT^*M$ and 
$$
\Phi(\RRR(\zeta))\circ T\RRR_{\zeta}\circ\Phi(\zeta)^{-1}=\left(\begin{array}{cc} -1 & 0\\
0 & 1 \end{array}\right)=:R
$$
We abbreviate the fixed locus of $R$ by
$$
V:=\Fix R=\R\x(0).
$$
Then the linearization $T\phi_Q^t$ of the flow of $X_Q$ gives a path in $\Sp(\R^{2})$
$$
\Psi_x(t):=\Phi(x(t))T\phi_Q^t(x(0))\Phi(x(0))^{-1},\quad t\in[0,T],
$$
and the {\em Robbin-Salamon index} of $(x,T)$ a solution of  \eqref{eq:characteristic chord eq with conormal boundary condition} is defined by
\beq\label{eq:Robbin-Salamon index}
\mu_{RS}(x,T):=\mu_{RS}({\Psi_x}(t)V,V).
\eeq
We remark that the Robbin-Salamon index does not depend on the choice of  vertical preserving symplectic trivialization, see \cite[Subsection 5.1]{Oh97} or \cite[Section 3]{APS08}.

\begin{Lemma}\label{prop:RS-index of the partner}
If $(x,T)$ is a solution of \eqref{eq:characteristic chord eq with conormal boundary condition}, it holds that
$$
\mu_{RS}(x,T)=\mu_{RS}(x_\RRR,T).
$$
\end{Lemma}
\begin{proof}
Since $x_\RRR(t)=\RRR x(T-t)$, we have
\bean
\Psi_{x_\RRR}(t)&=\Phi(x_\RRR(t))\circ T\phi_X^t(x_\RRR(0))\circ\Phi(x_\RRR(0))^{-1}\\
&=R\circ\Phi(x(T-t))\circ T\RRR_{x_\RRR(t)}\circ T\phi_X^t(x_\RRR(0))\circ T\RRR_{x_\RRR(0)}\circ\Phi(x(T))^{-1}\circ R\\
&=R\circ\Phi(x(T-t))\circ T\phi_X^{-t}(x(T))\circ\Phi(x(T))^{-1}\circ R.
\eea
We observe that 
$$
T\phi_X^{-t}(x(T))=T\phi_X^{T-t}(x(0))\circ [T\phi_X^T(x(0))]^{-1}
$$
and thus
$$
\Psi_{x_\RRR}(t)=R\circ \Psi_x(T-t)\circ [\Psi_x(T)]^{-1}\circ R.
$$
Since
$$
\mathrm{Id}=\Psi_x(T-t)\circ[\Psi_x(T-t)]^{-1}
$$
is homotopic to
$$
\Psi_x(0)\circ[\Psi_x(T-t)]^{-1}\#\Psi_x(T-t)\circ[\Psi_x(T)]^{-1},
$$
we obtain
$$
0=\mu_{RS}(\Psi_{x_\RRR}(t)V,V)+\mu_{RS}(R\circ [\Psi_x(T-t)]^{-1}\circ R V,V)
$$
Therefore we conclude that
\bean
\mu_{RS}(x_\RRR,T)&=\mu_{RS}(\Psi_{x_\RRR}(t)V,V)=-\mu_{RS}(R\circ [\Psi_x(T-t)]^{-1} R V,V)\\
&=\mu_{RS}([\Psi_x(T-t)]^{-1} R V,R V)=-\mu_{RS}(\Psi_x(T-t)V,V)\\
&=\mu_{RS}(\Psi_x(t)V,V)=\mu_{RS}(x,T).
\eea
\end{proof}
In \cite[Sectoion 6]{LZZ06}, Long-Zhang-Zhu proved the limit  
$$
\hat\mu_{RS}(x,T):=\lim_{m\to\infty}\frac{\mu_{RS}(x^{m},mT)}{m}.
$$
exists\footnote{Actually they used the $\mu_1$-index which is different from $\mu_{RS}$-index by a constant $n/2$, see \cite[Theorem 3.1]{LZ00}.} and
furthermore they proved the following identity 
$$
\hat\mu_{RS}(x,T)=\frac{1}{2}\hat\mu_{CZ}(x^2,2T)
$$
where 
$$
\hat\mu_{CZ}(x^2,2T):=\lim_{m\to\infty}\frac{\mu_{CZ}(x^{2m},2mT)}{m}.
$$
for a periodic solution $(x^2,2T)$, see \cite{Lon02}. In consequence, we obtain the following proposition which plays a crucial role in proving Theorem A.
\begin{Prop}\label{prop:index inequalities}
If $S$, the double cover of $\Sigma$, is dynamically convex, then
$$
\hat\mu_{RS}(x,T)>\frac{1}{2}.
$$
\end{Prop}
\begin{proof}
We note that $(x^4,4T)$ is a periodic Reeb orbit on $(S,\alpha)$. It can be readily verified that $\hat\mu_{CZ}(x^4,4T)>2$ by the assumption $\mu(x^4,4T)\geq 3$ and the index iteration formulas, see \cite[Section 8]{Lon02}.\footnote{ For instance if $(x^4,T^4)$ is an elliptic periodic orbit on $S$, $\hat\mu_{CZ}(x^4,4T)=\mu_{CZ}(x^4,4T)-1+\theta$ for some $\theta\in(0,1)$ and thus $\hat\mu_{CZ}(x^4,4T)>2$.} Thus we have 
$$
\lim_{m\to\infty}\frac{\mu_{RS}(x^{2m},2mT)}{m}=\hat\mu_{RS}(x^2,2T)=\frac{1}{2}\hat\mu_{CZ}(x^4,4T)>1
$$
and $\hat\mu_{RS}(x,T)>1/2$ is proved since the limit $\hat\mu_{RS}(x,T)$ also exists.
\end{proof}

%%%%%%%%%%%%%%%%%%%%%%%%%%%%%%%%%%%%%%%%%%%%%%%%%%%%%%%%%%%
\subsection{Proof of Theorem B}\quad\\[-1.5ex]

Since the Lagrangian Rabinowitz Floer homology \eqref{eq:computation of LagRFH} is different from the singular homology $\H_*(\Sigma\cap N^*S^1;\Z_2)$, there exists a nontrivial solution $(x,\eta)$ of 
$$
\dot x=\eta X_Q(x(t)),\quad x(t)\in\Sigma=Q^{-1}(\mu^2/2),\quad (x(0),x(1))\in\Fix\RRR\x\Fix\RRR,\quad \forall t\in[0,1].
$$
Assume on the contrary that $(x_\eta^2,2\eta)$ is the only (geometrically distinct) symmetric periodic orbit. Since $\mu_{RFH}(x,\eta)=\mu_{RS}(x_\eta,\eta)+1/2$, 
$$
\lim_{m\to\infty}\frac{\mu_{RFH}(x^{m},m\eta)}{m}>1/2
$$ 
due to Proposition \ref{prop:index inequalities} and thus there exist $r,\,N\in\N$, $N\geq 4$ such that
$$
\mu_{RFH}(x_\eta^m,m\eta)>\frac{m}{2}+\frac{m}{r},\quad \textrm{for all}\;\;\;m\geq N.
$$
We set 
$$
P_1:=\min\Big\{n\in\N\,\Big|\,\frac{N}{2}+\frac{N}{r}\leq n\Big\}
$$
and
$$
P_2:=\max\{\mu_{RFH}(x^\ell,\ell\eta)\,|\,\ell<N\}
$$
By definition, 
$$
\mu_{RFH}(x^m,m\eta)>P_1,\quad m\geq N.
$$
We abbreviate
$$
J=\#\{(x^\ell,\ell\eta)\,|\,\ell<N,\,\mu_{RFH}(x^\ell,\ell \eta)\in[P_1,P_2]\}\in\N\cup\{0\}.
$$
Since $\mu_{RFH}(x^m,m\eta)=\mu_{RFH}((x_\RRR)^m,m\eta)$ for all $m\in\Z$ due to Proposition \ref{prop:RS-index of the partner} and 
$$
\mu_{RFH}(x^{N+2rJ},(N+2rJ)\eta)>P_1+rJ+2J,
$$
we obtain
\bean
\dim_{\Z/2}\bigoplus_{q\in[P_1,P_1+rJ+2J]}\!\!\!\!\!\!\!\!\CF_q(\AA^H,f)
&\leq 2J
+2\#\{(x^N,N\eta),\dots (x^{N+2rJ-1},(N+2rJ-1)\eta)\}\\
&=2J+4rJ.
\eea
But on account of \eqref{eq:computation of LagRFH}, we have a contradictory inequality
$$
\dim_{\Z/2}\bigoplus_{q\in[P_1,P_1+rJ+2J]}\!\!\!\!\!\!\!\!\CF_q(\AA^H,f)\geq 4(rJ+2J+1)
$$
and this proves the theorem.
 \hfill$\square$

%%%%%%%%%%%%%%%%%%%%%%%%%%%%%%%%%%%%%%%%%%%%%%%%%%%%%%%%%%

\subsubsection*{Acknowledgments} {\em This work has its origin in inspiring  discussions with my advisor Urs Frauenfelder during the East Asia Symplectic Conference 2011 at KIAS. I am grateful to Urs Frauenfelder and Otto van Koert for fruitful discussions. I would like to thank Fabian Ziltener for inviting me to the geometry seminar at KIAS and for his helpful comments. I also thanks Will J. Merry for sending me his thesis.}

%%%%%%%%%%%%%%%%%%%%%%%%%%%%%%%%%%%%%%%%%%%%%%%%%%%%%%%%%%%%%%%%%

\end{document}